\documentclass{amsart}

\title{The Picard group of the category of $C_n$-equivariant stable homotopy theory}
\author{Vigleik Angeltveit}
\address{Mathematical Sciences Institute \\
Australian National University \\
Canberra, ACT 0200 \\
Australia}

\usepackage{amsxtra}
\usepackage{amsfonts}
\usepackage[latin1]{inputenc}
\usepackage{graphicx}
\usepackage{amsmath,amssymb,latexsym,amsthm,mathrsfs}
\usepackage[all]{xy}
\usepackage{pstricks}
\usepackage{verbatim}

\newtheorem{theorem}{Theorem}[section]
\newtheorem{thm}[theorem]{Theorem}
\newtheorem{lemma}[theorem]{Lemma}
\newtheorem{corollary}[theorem]{Corollary}

\newtheorem{prop}[theorem]{Proposition}

\theoremstyle{definition}
\newtheorem{defn}[theorem]{Definition}
\newtheorem{remark}[theorem]{Remark}
\newtheorem{example}[theorem]{Example}

\makeatletter
\let\c@equation\c@theorem
\makeatother
\numberwithin{equation}{section}

\newtheorem{lettertheorem}{Theorem}

\newcommand{\cA}{\mathcal{A}}              \newcommand{\cO}{\mathcal{O}}

  \newcommand{\bC}{\mathbb{C}}             
         \newcommand{\bZ}{\mathbb{Z}}

\newcommand{\sma}{\wedge} 

\newcommand{\xto}{\xrightarrow}

\newcommand{\pic}{\textnormal{Pic}}

\newcommand{\un}{\underline}
\newcommand{\Gsp}{Sp^G}
\newcommand{\Csp}{Sp^{C_n}}
\newcommand{\psp}{Sp^{C_p}}
\newcommand{\tr}{\textnormal{tr}}

\newcommand{\Spec}{\textnormal{Spec}}
\newcommand{\Mack}{\textnormal{Mack}}
\newcommand{\Mod}{\textnormal{Mod}}
\newcommand{\lsup}[1]{{}^{#1} \!\!}

\begin{document}

\begin{abstract}
For a finite group $G$, there is a map $RO(G) \to \pic(\Gsp)$ from the real representation ring of $G$ to the Picard group of $G$-spectra. This map is not known to be surjective in general, but we prove that when $G$ is cyclic this map is indeed surjective and in that case we describe $\pic(\Gsp)$ explicitly.

We also show that for an arbitrary finite group $G$ homology and cohomology with coefficients in a cohomological Mackey functor do not see the part of $\pic(\Gsp)$ coming from the Picard group of the Burnside ring. Hence these homology and cohomology calculations can be graded on a smaller group.
\end{abstract}

\maketitle

\section{Introduction}
The Picard group of the category of $G$-spectra is the group of isomorphism classes of invertible $G$-spectra (in the homotopy category) under the smash product. It has been studied by tom Dieck, Petrie, Fausk, Lewis, and May among others. See for example \cite{tD78b, tD78, tDP78, tD79, tD87, Ma01, FLM}.

There is a natural map $RO(G) \to \pic(\Gsp)$ from the real representation ring of $G$ to the Picard group of $G$-spectra, sending a virtual representation to the isomorphism class of the corresponding virtual representation sphere. We prove the following result:

\begin{lettertheorem} \label{t:A}
Let $C_n$ denote the cyclic group of order $n$. The map
\[
 RO(C_n) \to \pic(\Csp)
\]
is surjective, and
\[
 \pic(\Csp) \cong \prod_{\substack{d \mid n \\ d \neq 1, 2}} (\bZ/d)^\times/\{\pm 1\} \times \prod_{d \mid n} \bZ.
\]
\end{lettertheorem}

Here the first factor comes from $\pic$ of the Burnside ring of $C_n$ and the second factor is related to the dimension of the virtual representation.

The Picard group provides a natural home for grading equivariant stable homotopy groups. Making the notion of $\pic(\Gsp)$-graded homotopy groups of a $G$-spectrum precise depends on some choices, but Dugger \cite{Du14} proves that it is possible to make coherent choices and obtain well defined homotopy groups.

This means that rather than grading $C_n$-equivariant stable homotopy groups on $RO(C_n)$ we can grade them on the much smaller group $\pic(\Csp)$. When considering homology or cohomology with coefficients in a cohomological Mackey functor, further simplifications are possible.

\begin{lettertheorem} \label{t:B}
Let $G$ be a finite group, suppose $[X]$ is in the kernel of $d : \pic(\Gsp) \to C(G)$, and suppose $\un{M}$ is a cohomological Mackey functor. Then
\[
 X \sma H\un{M} \simeq H\un{M}.
\]
\end{lettertheorem}

Here $C(G)$ is the abelian group of functions from the set of conjugacy classes of subgroups of $G$ to $\bZ$ under pointwise addition. By \cite[Theorem 0.1]{FLM} we have an exact sequence
\[
 0 \to \pic(A(G)) \to \pic(\Gsp) \xto{d} C(G),
\]
where $A(G)$ is the Burnside ring of $G$. We can interpret Theorem \ref{t:B} as saying that homology and cohomology with coefficients in a cohomological Mackey functor do not see the part of $\pic(\Gsp)$ coming from $\pic(A(G))$.

\begin{example}
Consider the cyclic group $C_p$ for an odd prime $p$. Then $RO(C_p) \cong \bZ^{(p-1)/2}$. Theorem \ref{t:A} says that instead of grading $C_p$-equivariant stable homotopy groups on $\bZ^{(p-1)/2}$ we can grade them on $(\bZ/p)^\times/\{\pm 1\} \times \bZ^2$. If $p=3$ there is no difference, but for large $p$ the group $\pic(\psp)$ is much smaller.

Moreover, Theorem \ref{t:B} says that homology or cohomology groups with coefficients in a cohomological Mackey functor can be graded on just $\bZ^2$ regardless of the prime.
\end{example}

\begin{remark}
In the existing literature the reduction from grading on $RO(G)$ to grading on a smaller group is usually done after localizing at $p$. See for example \cite[Prop.\ 2.1]{HHR17} (and the sentence just before that, which emphasizes the $p$-local context).
\end{remark}

\subsection{Conventions}
We assume that the reader is familiar with the theory of Mackey functors and with equivariant stable homotopy theory. Any category of ``genuine'' $G$-equivariant spectra with the usual homotopy category will do. For concreteness, we will work in the category of ortogonal $G$-spectra as described in \cite{HHR}.

We will use notation like $\un{M}$ for a Mackey functor, and write $\un{M}(G/H)$ for the value of $\un{M}$ at the subgroup $H$. Given subgroups $K, L \leq H \leq G$ we have a transfer map $\tr_K^H : \un{M}(G/K) \to \un{M}(G/H)$ and a restriction map $R^H_L : \un{M}(G/H) \to \un{M}(G/L)$. The composite $R^H_L \circ \tr_K^H$ can be computed using a ``double coset formula''. Each $\un{M}(G/H)$ also has an action of $N_H(G)/H$, and there are conjugation maps between $\un{M}(G/H)$ and $\un{M}(G/gHg^{-1})$, but these will play no role in the current paper. See for example \cite{We00} for details.

\subsection{Organization}
In Section \ref{s:picBurnside} we compute $\pic(A(C_n))$ and in Section \ref{s:imd} we complete the proof of Theorem \ref{t:A}.

In Section \ref{s:picMackey} we describe explicit representatives of $\pic(A(C_n))$, and in Section \ref{s:pi0calc} we describe explicit $C_n$-spectra representing $\pic(\Csp)$.

In Section \ref{s:coherence} we explain how our choices from the previous two sections give a $\pic(\Csp)$-trivialisation of $\Csp$ in the sense of Dugger \cite{Du14}, and we explain in what sense $\pic(\Csp)$-graded stable homotopy groups of a homotopy commutative $G$-spectrum are ``graded commutative''.

Finally, in Section \ref{s:cohomMackey} we prove Theorem \ref{t:B}. This turns out to boil down to a calculation of the Picard group of $H\un{\bZ}$.

\subsection{Acknowledgements}
The author would like to thank Mike Hill and Anna Marie Bohmann for interesting discussions, and James Borger and Anand Deopurkar for help with the algebraic geometry in Section \ref{s:picBurnside}.

\section{The Picard group of the Burnside ring} \label{s:picBurnside}
Much of the material in this section can be found in \cite{tD78} and \cite{tDP78}, using slightly different language. But we have been unable to find the exact calculations we need in the literature. Let $G$ be a finite group. Suppose $G$ has $r$ conjugacy classes of subgroups represented by $H_1,\ldots,H_r$. If we need to be explicit about it we will let $H_1 = G$ and $H_r = e$.

We let $C(G) \cong \bZ^r$ denote the abelian group of functions from the set of conjugacy classes of subgroups to $\bZ$ under pointwise addition. In fact it will be convenient to consider $C(G)$ as a ring under pointwise addition and multiplication.

Let $A(G)$ denote the Burnside ring of $G$, which as an abelian group is free on $[G/H_1],\ldots,[G/H_r]$. The multiplication is given by decomposing $G/H_i \times G/H_j$ into orbits, and $[G/H_1] = [G/G]$ is the multiplicative identity. Let
\[
 w : A(G) \to C(G)
\]
be the Burnside ghost map, given on a finite $G$-set $X$ by $X \mapsto \big( H_i \mapsto |X^{H_i}| \big)$,
or $X \mapsto \big(|X^{H_1}|,\ldots,|X^{H_r}|\big)$. This map is injective, and the ring structure on $A(G)$ is determined by requiring that $w$ is a ring map.

\begin{remark}
Recall from \cite[Corollary 2]{DS88} that the Burnside ring $A(G)$ is isomorphic to the ring $W_G(\bZ)$ of $G$-Witt vectors of $\bZ$. If we represent an element in $A(G)$ by $(a_1,\ldots,a_r)$ where $a_i$ denotes the coefficient of $[G/H_i]$ then the Burnside ghost map is given by $(a_1,\ldots,a_r) \mapsto \langle x_1,\ldots,x_r\rangle$ with
\[
 x_i = \sum \big|(G/H_j)^{H_i}\big| a_j, 
\]
with the sum being over all $j$ so that $H_i$ is subconjugate to $H_j$. This can be compared to the Witt vector ghost map, which is given by the same formula except with
\[
 x_i = \sum \big|(G/H_j)^{H_i}\big| a_j^{[H_i:H_j]}.
\]
(Here $[H_i:H_j]$ denotes the index of the appropriate conjugate of $H_i$ in $H_j$.)
\end{remark}

Let $X=\Spec(C(G))$, let $Y = \Spec(A(G))$, and let $f : X \to Y$ be $\Spec$ of the Burnside ghost map. If we complete at a prime $p$ which does not divide $|G|$, $f$ is an isomorphism. Now we can consider the following short exact sequence of sheaves of abelian groups on $Y$, with $Q$ defined as the quotient:
\[
 0 \to \cO_Y^\times \to f_*(\cO_X^\times) \to Q \to 0.
\]
The sheaf $Q$ is concentrated at those $p$ which divide $|G|$. Because $f$ is finite, each $Q^\wedge_p$ is finite and it follows that the global sections of $Q$ is a finite abelian group.

\begin{thm}[tom Dieck {\cite[Section 4]{tD78}}]
For any finite group $G$, the Picard group of $A(G)$ sits in a $4$-term exact sequence
\[
 0 \to A(G)^\times \to (\bZ^r)^\times \to \Gamma(Y,Q) \to \pic(A(G)) \to 0.
\]
\end{thm}

\begin{corollary}
For any finite group $G$ the Picard group $\pic(A(G))$ is finite.
\end{corollary}

tom Dieck goes on to discuss a procedure for computing $\Gamma(Y,Q)$ when $G$ is abelian, but we can actually compute $\Gamma(Y,Q)$ in general. Because $Q$ is concentrated at a finite set of primes,
$\Gamma(Y,Q) \cong \prod_p Q_p$ is a product of stalks. And the stalk $Q_p$ can be computed as the cokernel
\[
 (A(G) \otimes \bZ_p)^\times \xto{w} (\bZ_p^\times)^r \to Q_p.
\]

Now we put this together with the ``Cauchy-Frobenius-Burnside'' relations, which say that $x$ is in the image of the Burnside ghost map if and only if, for each $i$, the congruence
\[
 \sum_{gH_i \in N_G(H_i)/H_i} x(H_i,g) \equiv 0 \qquad \mod [N_G(H_i):H_i]
\]
holds, where $x(H_i,g)$ is the value of $x$ on the subgroup generated by $H_i$ and $g$. From this we conclude that
\[
 \Gamma(Y,Q) \cong \prod_{i=1}^r \big(\bZ/[N_G(H_i):H_i]\big)^\times,
\]
where the product is over the conjugacy classes of subgroups of $G$ and we interpret $(\bZ/1)^\times$ as the trivial group.

At this point we specialise to $G=C_n$ and compute the first factor $\pic(A(C_n))$ from Theorem \ref{t:A}:

\begin{thm} \label{t:picBurnside}
The Picard group of the Burnside ring $A(C_n)$ is given by
\[
 \pic(A(C_n)) \cong \prod_{\substack{ d \mid n \\ d \neq 1, 2}} \big( \bZ/d \big)^\times/\{\pm 1\}.
\]
\end{thm}

\begin{proof}
We have one subgroup $C_{n/d} \leq C_n$ for each $d \mid n$, and we get
\[
 \Gamma(Y, Q) \cong \prod_{d \mid n} \big( \bZ/(C_n/C_{n/d}) \big)^\times \cong \prod_{d \mid n} \big( \bZ/d \big)^\times.
\]
We can omit the $d=1$ and $d=2$ factors, as these are trivial.

If $n$ is odd we have $A(C_n)^\times \cong \{\pm 1\}$ with the $-1$ represented by $-[C_n/C_n]$. The image of this group of units under the Burnside ghost map is the diagonal $\{\pm 1\}$, so if $n$ has $r$ divisors the map $A(C_n)^\times \to (\bZ^r)^\times$ is the diagonal map $\{\pm 1\} \to \{\pm 1\}^r$. The cokernel of this map is isomorphic to $\{\pm 1\}^{r-1}$, and $\pic(A(C_n)) \cong \Gamma(Y, Q)/\{\pm 1\}^{r-1}$ with each $\{\pm 1\}$ acting on a single $(\bZ/d)^\times$ for $d \neq 1$, This gives the result for $n$ odd.

If $n$ is even we have $A(C_n)^\times \cong \{\pm 1\}^2$ with generators represented by $-1 = -[C_n/C_n]$ and $\tau = [C_n/C_{n/2}] - [C_n/C_n]$. The image of $-1$ under the Burnside ghost map is $C_d \mapsto -1$ for all $d \mid n$, and the image of $\tau$ is $C_d \mapsto \begin{cases} 1 \quad & \textnormal{if $d \mid n/2$} \\ -1 \quad & \textnormal{if $d \nmid n/2$} \end{cases}$

The cokernel of this map is isomorphic to $\{\pm 1\}^{r-2}$ with each $\{\pm 1\}$ acting on a single $(\bZ/d)^\times$ for $d \neq 1, 2$. This gives the result for $n$ even.
\end{proof}

This completes the description of the factor $\prod\limits_{\substack{d \mid n \\ d \neq 1, 2}} (\bZ/d)^\times/\{\pm 1\}$ in the expression for $\pic(\Csp)$ in Theorem \ref{t:A}.

One can, with some difficulty, give an explicit description of an invertible $A(C_n)$-module representing each element of $\pic(A(C_n))$. But we will wait until Section \ref{s:picMackey} to do so, as the perspective of Mackey functors makes the description more intuitive. In Section \ref{s:pi0calc} we find an explicit $C_n$-spectrum representing each element of $\pic(A(C_n))$.

\section{The image of $d : \pic(\Gsp) \to C(G)$} \label{s:imd}
Recall from \cite[Definition III.5.1]{tD87} that $f \in C(G)$ is a \emph{Borel-Smith function} if it satisfies a certain list of congruences.

For any invertible $X \in \Gsp$, $d(X)$ is a Borel-Smith function, and by \cite[Theorem III.5.4]{tD87} the map $d \circ jO : RO(G) \to C(G)$ surjects onto the subgroup of Borel-Smith functions when $G$ is nilpotent. In particular this holds when $G$ is abelian.

Because the image of $d$ is a subgroup of a free abelian group, it is itself free and after choosing a basis we get a short exact sequence
\[
 0 \to \pic(A(G)) \to \pic(\Gsp) \to \bZ^r \to 0
\]
for some $r$. Here $r$ could be smaller than the number of conjugacy classes of subgroups of $G$. When $G$ is nilpotent $r$ is equal to the number of conjugacy classes of cyclic subgroups of $G$.

This can be completed to the following diagram with exact rows:
\[ \xymatrix{
 0 \ar[r] & RO_0(G) \ar[r] \ar[d] & RO(G) \ar[r] \ar[d] & \overline{RO}(G) \ar[r] \ar[d]^\cong & 0 \\
 0 \ar[r] & \pic(A(G)) \ar[r] & \pic(\Gsp) \ar[r] & \bZ^r \ar[r] & 0
} \]

If we choose a splitting we get
\[
 \pic(\Gsp) \cong \pic(A(G)) \times \bZ^r,
\]
and every element in the second factor (or more precisely every element $(1,b)$ where $1 \in \pic(A(G))$ is the identity element) is represented by a virtual representation sphere.

We once again specialise to $G=C_n$. Recall that the irreducible real representations of $C_n$ are the trivial representation $1$, the sign representation $\sigma$ if $n$ is even, and the $2$-dimensional representations $\lambda(i)$ for $i=1,\ldots,\lfloor \frac{n-1}{2} \rfloor$ with $z \in C_n \subset \bC$ acting by multiplication by $z^i$.

First, $C(C_n)$ is a free abelian group of rank equal to the number of divisors of $n$. The subgroup of Borel-Smith functions is the finite index subgroup generated by the functions $\{f_d\}_{d \mid n}$ defined as follows: We have $f_1(H) = 1$ for all $h$, and if $2 \mid n$ we have $f_2(H) = 1$ for $H \leq C_{n/2}$ and zero otherwise. For $d \neq 1, 2$ we have $f_d(H) =2$ for $H \leq C_{n/d}$ and zero otherwise. We will write $\prod\limits_{d \mid n} \bZ$ for the subgroup of Borel-Smith functions, and observe that these are exactly the dimension functions of the irreducible $C_n$-representations. (But note that multiple irreducible representations can have the same dimension function.) We then have a short exact sequence
\[
 0 \to \pic(A(C_n)) \to \pic(Sp^{C_n}) \to \prod_{d \mid n} \bZ \to 0.
\]

Now we go ahead and choose the following splitting $\prod\limits_{d \mid n} \bZ \to \pic(Sp^{C_n})$: For $b = (b_d) \in \prod\limits_{d \mid n} \bZ$ we let
\[
 Y_b = S^{b_1} \sma S^{b_2 \sigma} \sma \bigwedge_{\substack{d \mid n \\ d \neq 1, 2}} S^{b_d \lambda(n/d)}.
\]
Here we omit the smash factor $S^{b_2 \sigma}$ if $n$ is odd.

This completes the description of the factor $\prod\limits_{d \mid n} \bZ$ in the expression for $\pic(\Csp)$ in Theorem \ref{t:A}, and together with Section \ref{s:picBurnside} this completes the proof of Theorem \ref{t:A}.

\section{The Picard group of Mackey functors} \label{s:picMackey}
There is another natural notion of a Picard group in $G$-equivariant stable homotopy theory. Recall, e.g.\ from \cite{We00}, that the category of $G$-Mackey functors is a symmetric monoidal category under the box product $\Box$ with unit the Burnside Mackey functor $\un{\cA}_G$. We can then consider the group $\pic(\Mack_G)$ of isomorphism classes of invertible $G$-Mackey functors under $\Box$.

\begin{thm} \label{t:picBurnsidepicMackey}
For any finite group $G$ we have a natural isomorphism $\pic(A(G)) \cong \pic(\Mack_G)$.
\end{thm}

\begin{proof}
We have an adjunction
\[
 F : \Mod_{A(G)} \rightleftarrows \Mack_G : U.
\]
Here $F$ is defined by $F(M)(G/H) = A(H) \otimes_{A(G)} M$, with restriction and transfer maps induced by the restriction and transfer maps between Burnside rings. The functor $U$ is defined by $U(\un{N}) = \un{N}(G/G)$, with $A(G)$-module structure given by $[G/H] \cdot x = \tr_H^G(R^G_H(x))$.

The composite $U \circ F$ is naturally isomorphic to the identity, while $F \circ U(\un{N})$ picks out the sub-Mackey functor of $\un{N}$ with $F \circ U(\un{N})(G/H)$ generated by $\tr_K^H(R^G_K(x))$ for $x \in \un{N}(G/G)$ and $K \leq H$.

Both functors map invertible objects to invertible objects, and $F \circ U$ is naturally isomorphic to the identity when restricted to invertible Mackey functors, so $F$ and $U$ induce inverse isomorphisms on Picard groups.
\end{proof}

Next we describe a certain family of $C_n$-Mackey functors, starting with the Burnside Mackey functor $\un{\cA} = \un{\cA}_{C_n}$. As abelian groups we have $\un{\cA}(C_n/C_m) = A(C_m) = \bZ\{x^m_d\}_{d \mid m}$, with $x^m_d$ representing the finite $C_m$-set $C_m/C_{m/d}$ or cardinality $d$. The transfer maps are given by
\[
 \tr^k : A(C_m) \to A(C_{km}) \qquad \tr^k(x^m_q) = x^{km}_{kq},
\]
and the restriction maps are determined by
\[
 R^k(x^n_1) = x^{n/k}_1
\]
for all $k \mid n$. The action of $C_n/C_m$ on $\un{\cA}(C_n/C_m)$ is trivial.

\begin{defn}
Suppose $a = (a_d) \in \prod\limits_{\substack{d \mid n \\ d \neq 1}} \bZ$ satisfies $a_e \mid a_d$ for $e \mid d$, and $a_d \neq 0$ for each $d$. Let $\lsup{a} \un{\cA}$ denote the Mackey functor defined in the same way as $\un{\cA}$, except with restriction maps determined by
\[
 R^k(x^n_1) = a_k x^{n/k}_1.
\]
\end{defn}

\begin{remark}
The remaining restriction maps are determined as follows: For a generator $x^{n/k}_1$ we use that $R^{k\ell} = R^\ell \circ R^k$ to conclude that $R^\ell(x^{n/k}_1) = \frac{a_{k\ell}}{a_k} x^{n/k\ell}_1$. Any other generator $x^m_d$ for $d > 1$ is in the image of some transfer map, and in this case $R^k(x^m_d)$ is determined by how restriction and transfer maps commute.
\end{remark}

In general the box product of Mackey functors is difficult to compute, but for this particular family of Mackey functors we have the following:

\begin{lemma}
Given Mackey functors $\lsup{a} \un{\cA}$ and $\lsup{b} \un{\cA}$ as defined above we have
\[
 \lsup{a} \un{\cA} \Box \lsup{b} \un{\cA} \cong \lsup{ab} \un{\cA},
\]
where $(ab)_d = a_d b_d$.
\end{lemma}

\begin{proof}
We have
\[
 \lsup{a} \un{\cA} \Box \lsup{b} \un{\cA}(C_n/C_m) = \lsup{a} \un{\cA}(C_n/C_m) \otimes \lsup{b} \un{\cA}(C_n/C_m) \oplus Im(\tr)/\sim,
\]
where the equivalence relation is given by Frobenius reciprocity. If we denote the basis for $\lsup{a} \un{\cA}(C_n/C_m)$ by $\{x^m_q\}$ and the basis for $\lsup{b} \un{\cA}(C_n/C_m)$ by $\{y^m_q\}$, Frobenius reciprocity tells us that
\begin{eqnarray*}
 x^m_q \otimes y^m_r \sim \tr^q(x^{m/q}_1 \otimes R^q(y^m_r)) \qquad \textnormal{for $q > 1$} \\
 x^m_q \otimes y^m_r \sim \tr^r(R^r(x^m_q) \otimes y^{m/r}_1)) \qquad \textnormal{for $r > 1$} \\
\end{eqnarray*}
Hence $\lsup{a} \un{\cA} \Box \lsup{b} \un{\cA}(C_n/C_m) \cong \bZ\{x^m_1 \otimes y^m_1\} \oplus Im(\tr)$, and because $R^k(x^m_1 \otimes y^m_1) = R^k(x^m_1) \otimes R^k(y^m_1)$ the result follows.
\end{proof}

We would like to understand when $\lsup{a} \un{\cA}$ is invertible and when $\lsup{a} \un{\cA} \cong \lsup{b} \un{\cA}$. To do this we can study what happens when we do a change of basis, for example by replacing $x^n_1$ by $x^n_1 + \sum\limits_{d > 1} c_d x^n_d$ for certain $c_d \in \bZ$.

But before we do that, we restrict our attention to those $a$ for which $a_r a_s \mid a_{rs}$ for all relatively prime $r$ and $s$.

\begin{defn}
Let $a = (a_d) \in \prod\limits_{\substack{d \mid n \\ d \neq 1}} \bZ$ with each $a_d \neq 0$. Let $\un{\cA}^a$ denote the Mackey functor defined in the same way as $\un{\cA}$, except with restriction maps determined by
\[
 R^k(x^n_1) = \big( \prod_{d \mid k} a_d \big) x^{n/k}_1.
\]
\end{defn}

Any Mackey functor $\un{\cA}^a$ is of the form $\lsup{b} \un{\cA}$ for $b_k = \prod\limits_{d \mid k} a_d$, but the reverse is not true. Note that we still have
\[
 \un{\cA}^a \Box \un{\cA}^b \cong \un{\cA}^{ab}.
\]

\begin{prop}
Given tuples $a=(a_d)$ and $b=(b_d)$ we have the following:
\begin{enumerate}
 \item Suppose $a_d \equiv b_d \mod d$ for each $d$. Then $\un{\cA}^a \cong \un{\cA}^b$.
 \item Suppose $a_d = \pm b_d$ for each $d$. Then $\un{\cA}^a \cong \un{\cA}^b$.
 \item Suppose $\un{\cA}^a \cong \un{\cA}^b$. Then $a_d \equiv \pm b_d \mod d$ for each $d$.
\end{enumerate}
\end{prop}

It follows that $\un{\cA}^a$ is invertible if and only if each $a_d$ is invertible in $\bZ/d$, and in that case the isomorphism class of $\un{\cA}^a$ is determined by the equivalence class of $a_d$ in $(\bZ/d)^\times/\{\pm 1\}$ for each $d$. By Theorem \ref{t:A} and Theorem \ref{t:picBurnsidepicMackey} we know how many isomorphism classes of invertible $C_n$-Mackey functors we are supposed to have, and because the above result describes the same number of invertible $C_n$-Mackey functors it follows that these are all of them.

\begin{proof}
We perform an appropriate change of basis. If we start with $\un{\cA}^a$ and we want to add $k$ to $a_k$, we replace $x^n_1$ by
\[
 y^n_1 = x^n_1 + \big( \prod\limits_{d \mid k, d \neq k} a_d \big) x^n_k
\]
Then
\[
 R^k(y^n_1) = \prod\limits_{d \mid k} a_k x^{n/k}_1 + d \big( \prod\limits_{d \mid k, d \neq k} a_d \big) x^{n/k}_1 = (a_k+k) \big( \prod\limits_{d \mid k, d \neq k} a_d \big) x^{n/k}_1.
\]
Moreover, we change the basis for $\un{\cA}^a(C_n/C_m)$ for each $m$ so that $R^{n/m}(y^n_1)$ is a multiple of $y^m_1$ and so that the transfer maps are given by the same formula as before. That proves the first part. For the second part, we do the change of basis $y^m_1 = \pm x^m_1$ for appropriate $m$.

For the last part, note that an isomorphism can be interpreted as a change of basis, and that any change of basis from $\{x^m_d\}$ to $\{y^m_d\}$ with the two properties that the description of the transfer maps remains intact and that $R^k(y^n_1)$ is a multiple of $y^{n/k}_1$ must be a combination of the changes of basis described above.
\end{proof}

\begin{remark}
For any $n$ there is one change of basis that does not change $a=(a_d)$, namely replacing each $x^m_q$ by $-x^m_q$. This corresponds to the unit $-1=[C_n/C_n]$ in $A(C_n)$.

If $n$ is even there is one other change of basis that does not change $a=(a_d)$, given by $y^n_1 = -x^n_1 + x^n_2$. This corresponds to the unit $\tau = [C_n/C_{n/2}] - [C_n/C_n]$ in $A(C_n)$.
\end{remark}

Any element of $\pic(\Mack_{C_n})$ is represented by $\un{\cA}^a$ with $1 \leq a_d < \frac{d}{2}$ in a unique way. Given $a=(a_d) \in \prod\limits_{\substack{d \mid n \\ d \neq 1, 2}} \bZ$ with each $a_d$ relatively prime to $d$, let $\overline{a}$ denote the representative of the equivalence class of $a$ in $\prod\limits_{\substack{d \mid n \\ d \neq 1, 2}} (\bZ/d)^\times/\{\pm 1\}$ with $1 \leq a_d < \frac{d}{2}$.

Given $a$, we have multiple isomorphisms $\un{\cA}^a \to \un{\cA}^{\overline{a}}$. The set of such isomorphisms is in bijection with $Aut(\un{\cA})$. As observed in the proof of Theorem \ref{t:picBurnside}, for $n$ odd we have $Aut(\un{\cA}) \cong \bZ/2$ and for $n$ even we have $Aut(\un{\cA}) \cong \bZ/2 \times \bZ/2$.

We can pick out a preferred isomorphism as follows. Let us denote the above basis for $\un{\cA}^a$ by $\{x^m_q\}$ and the basis for $\un{\cA}^{\overline{a}}$ by $\{y^m_q\}$. If $n$ is odd we insist that $x^1_1 \mapsto y^1_1$ (as opposed to $-y^1_1$) and this uniquely determines the isomorphism. If $n$ is even we also insist that the isomorphism sends $x^n_1$ to $y^n_1 + \sum\limits_{d \geq 2} c_d y^n_d$ (as opposed to $-y^n_1 + \sum\limits_{d \geq 2} c_d y^n_d$).

We will get back to this in Section \ref{s:coherence}.

\section{A $\un{\pi}_0$ calculation} \label{s:pi0calc}
In this section we find explicit representatives for the $\pic(A(C_n))$ factor of $\pic(\Csp)$ from Theorem \ref{t:A}.

\begin{thm}
Let $X = S^{\lambda(n/d) - \lambda(a n/d)}$ for some $a$ relatively prime to $d$. Then $\un{\pi}_n X = 0$ for any integer $n \neq 0$ and $\un{\pi}_0 X \cong \un{\cA}^{\hat{a}}$ for $\hat{a}=(a_e)$ with $a_e = \begin{cases} a \quad & e=d \\ 1 \quad & \textnormal{otherwise} \end{cases}$
\end{thm}

\begin{proof}
Because $X$ is invertible and $[X]$ is in the kernel of $d : \pic(\Csp) \to C(C_n)$, we already know that $\un{\pi}_0 X$ is an invertible $\un{\cA}$-module and that $\un{\pi}_n X = 0$ for $n \neq 0$. By the equivariant Hurewicz theorem \cite[Theorem 2.1]{Le92} it then suffices to show that $\un{H}_0(X;\un{\cA}) \cong \un{\cA}^{\hat{a}}$.

For ease of notation, let $d'=n/d$. We give $S^{\lambda(d')}$ a $C_n$-CW structure with a $0$-cell $C_n/C_n$, a $1$-cell $C_n/C_{d'}$, and a $2$-cell $C_n/C_{d'}$. Similarly, we give $S^{-\lambda(a_d d')}$ a $C_n$-CW structure with a $0$-cell $C_n/C_n$, a $(-1)$-cell $C_n/C_{d'}$, and a $(-2)$-cell $C_n/C_{d'}$.

We then get a $C_n$-CW structure on the smash product by taking a product of the two $C_n$-CW structures, and we get a chain complex which computes $\un{H}_*(X; \un{\cA})(C_n/C_n)$ by taking the total complex of the following double complex:
\[ \xymatrix{
 \un{\cA}(C_n/C_n) \ar[d] & \un{\cA}(C_n/C_{d'}) \ar[d] \ar[l] & \un{\cA}(C_n/C_{d'}) \ar[d] \ar[l] \\
 \un{\cA}(C_n/C_{d'}) \ar[d] & \un{\cA}(\coprod\limits_d C_n/C_{d'}) \ar[d] \ar[l] & \un{\cA}(\coprod\limits_d C_n/C_{d'}) \ar[d] \ar[l] \\
 \un{\cA}(C_n/C_{d'}) & \un{\cA}(\coprod\limits_d C_n/C_{d'}) \ar[l] & \un{\cA}(\coprod\limits_d C_n/C_{d'}) \ar[l]
} \]
Here we have identified $C_n/C_{d'} \times C_n/C_{d'}$ with $\coprod\limits_d C_n/C_{d'}$. If we want to compute $\un{H}_*(X; \un{\cA})(C_n/C_e)$ we can take the product of each of these $C_n$-sets with $C_n/C_e$ before evaluating.

We know what every entry is, and we can compute all the maps. Evaluated on $C_n/C_n$ we get the following:
\[ \xymatrix{
 \bZ\{x^n_q\} \ar[d]_{R^d} & \bZ\{x^{d'}_q\} \ar[d]_\Delta \ar[l]_-{\tr^d} & \bZ\{x^{d'}_q\}\ar[d]_\Delta \ar[l]_-0 \\
 \bZ\{x^{d'}_q\} \ar[d]_0 & \bZ\{x^{d'}_q(j)\}_{0 \leq j < d} \ar[d]_{1-sh_a} \ar[l]_-\nabla & \bZ\{x^{d'}_q(j)\}_{0 \leq j < d} \ar[d]_{1-sh_a} \ar[l]_-{1-sh} \\
 \bZ\{x^{d'}_q\} & \bZ\{x^{d'}_q(j)\}_{0 \leq j < d} \ar[l]_-\nabla & \bZ\{x^{d'}_q(j)\}_{0 \leq j < d} \ar[l]_-{1-sh}
} \]
Here we consider all $q$ which divide $n$ in the upper left corner and all $q$ which divide $d'$ in the rest of the diagram. The maps labelled $\nabla$ are fold maps, the maps labelled $\Delta$ are diagonal maps, $sh$ shifts the index $j$ by $1$ (modulo $d$) and $sh_a$ shifts the index $j$ by $a$ (modulo $d$).

We can compute the homology of the total complex using a small spectral sequence. If we first take homology with respect to the horizontal maps and then with respect to the vertical maps we find that $\un{H}_0(X; \un{\cA}) \cong \prod\limits_{q \mid n} \bZ$ as expected, with two types of generators:

\noindent
\textbf{Type 1}: If $d \nmid q$, the generator $x^n_q \in \un{\cA}(C_n/C_n)$ from the upper left corner survives the spectral sequence.

\noindent
\textbf{Type 2}: If $d \mid q$, the diagonal entry $\sum\limits_{0 \leq j < d} x^{d'}_{q'}(j)$ for $q'=q/d$ from the lower right corner survives the spectral sequence.

The type 2 generators are cycles in the double complex, but the type 1 generators are not because $R^d$ applied to one of the type 1 generators is nozero.

It suffices to consider the generator $x^n_1$ in the upper left corner. We can find an actual generator represented by $x^n_1$ by considering the following zig-zag:
\[ \xymatrix{
 x^n_1 \ar@{|->}[d] & & \\
 x^{d'}_1 & x^{d'}_1(1) \ar@{|->}[l] & \\
 & x^{d'}_1(0) - x^{d'}_1(a) \ar@{|->}[u] & \sum\limits_{0 \leq j < a} x^{d'}_1(j) \ar@{|->}[l]
} \]
In other words, the generator represented by $x^n_1$ in the upper left hand corner in the spectral sequence is the sum of $x^n_1$ in the upper left, $x^{d'}_1(1)$ in the middle, and $\sum\limits_{0 \leq j < a} x^{d'}_1(j)$ in the lower right hand corner.

Now we need to compute $R^e$ of this generator. If $d \nmid e$ then $R^e(x^n_1) = x^{n/e}_1$ in the upper left corner still survives the spectral sequence. So in this case $R^e$ does not pick up any additional factors.

If $d \mid e$ then $C_n/C_{d'} \times C_n/C_e \cong \coprod\limits_d C_n/C_e$ and the upper left transfer map becomes a fold map. Hence $R^e(x^n_1) = x^{n/e}_1$ in the upper left corner is in the image of the horizontal transfer map, and we are left with $R^e(\sum\limits_{0 \leq j < a} x^{d'}_1(j))$ in the lower right hand corner. Because we have $a$ terms we now pick up a factor of $a$. This completes the calculation, and hence the proof.
\end{proof}

Given $a=(a_d)$ we can then consider the smash product $X_a = \bigwedge\limits_{d \mid n} S^{\lambda(n/d) - \lambda(a_d n/d)}$ and we find that $\un{\pi}_0 X \cong \un{\cA}^a$.

\section{Coherence} \label{s:coherence}
In \cite{Du14}, Dugger explains how it is possible to make coherent choices and obtain well defined $\pic(\Gsp)$-graded homotopy groups of a $G$-spectrum as well as an associative and ``graded commutative'' multiplication on $\un{\pi}_\star(S^0)$.

Here we put ``graded commutative'' in quotation marks because we have to keep track of units in $A(G)$ rather than just $\pm 1$.

Let us mix additive and multiplicative notation in the Picard group by writing each $(\bZ/d)^\times/\{\pm 1\}$ multiplicatively and each $\bZ$ additively.

For each
\[
 (a,b) \in \pic(\Csp) \cong \prod_{\substack{d \mid n \\ d \neq 1, 2}} (\bZ/d)^\times/\{\pm 1\} \times \prod_{d \mid n} \bZ
\]
we define $\overline{a} \in \prod\limits_{\substack{d \mid n \\ d \neq 1, 2}} \bZ$ and $X_{\overline{a}}$ as in Section \ref{s:pi0calc}. We also define $Y_b$ as in Section \ref{s:imd}.

As explained in \cite[Section 7]{Du14} we then need to define coherent maps
\[
 X_{\overline{aa'}} \sma Y_{b+b'} \to (X_{\overline{a}} \sma Y_b) \sma (X_{\overline{a'}} \sma Y_{b'}
\]

The coherence maps for the $Y_b$ are easiest: As in \cite[Section 6]{Du14}, they are determined by the canonical maps $S^0 \to S^{-V} \sma S^V$ for $V$ one of $1$, $\sigma$, and $\lambda(n/d)$ for $d \mid n$ and $d \neq 1, 2$.

The coherence maps for the $X_{\overline{a}}$ are slightly more complicated, but we already dealt with that in Section \ref{s:picMackey}. We specified a particular isomorphism $\un{\cA}^a \Box \un{\cA}^{a'} \to \un{\cA}^{aa'}$ for each pair $(a,a')$ and a particular isomorphism $\un{\cA}^a \to \un{\cA}^{\overline{a}}$ for each $a$. These then specify particular isomorphisms $X_a \sma X_{a'} \to X_{aa'}$ and $X_a \to X_{\overline{a}}$.

With these choices, we then get particular maps
\[
 X_{\overline{a}} \sma X_{\overline{a'}} \to X_{\overline{a} \cdot \overline{a'}} \to X_{\overline{aa'}}.
\]
For coherence we have to check that the two maps
\[
 X_{\overline{a}} \sma X_{\overline{a'}} \sma X_{\overline{a''}} \to X_{\overline{aa'a''}}
\]
agree. But this is clear, because on $\un{\pi}_0$ both send $x^1_1 \otimes y^1_1 \otimes z^1_1$ to $w^1_1$ and, if $n$ is even, both send $x^n_1 \otimes y^n_1 \otimes z^n_1$ to $w^n_1 + \sum\limits_{d \geq 2} c_d w^n_d$.

In general, the $G$-equivariant stable homotopy groups of spheres, or of a homotopy commutative $G$-ring spectrum, are graded commutative with units in the Burnside ring. The same is true for $\pic(\Gsp)$-graded homotopy groups:

\begin{thm}
With notation as above, the $\pic(\Csp)$-graded stable homotopy groups of spheres are graded commutative in the sense that if $x \in \pi_{a,b}^G(S)$ and $y \in \pi_{a',b'}^G(S)$ then
\[
 xy = (-1)^{b_1 b_1'} yx
\]
if $n$ is odd and
\[
 xy = (-1)^{b_1 b_1'} (-\tau)^{b_2 b_2'}
\]
for $n$ even, where $\tau = [C_n/C_{n/2}] - [C_n/C_n]$.
\end{thm}

\begin{proof}
This is folklore for $RO(C_n)$-graded homotopy groups. See e.g.\ \cite[Prop.\ 1.18]{Du14} for a related result. In the $RO(C_n)$-equivariant context we can compute the ``signs'' in the units of $A(C_n)$ by computing the trace of the identity element on each $S^V$ for $V$ irreducible. For $S^1$ with a trivial action we pick up the usual $-1$.

If $V=\lambda(i)$ for some $i=1,\ldots,\lfloor \frac{n-1}{2} \rfloor$ then the composite
\[
 S^0 \xto{\eta} S^{-V} \sma S^V \xto{t} S^V \sma S^{-V} \xto{\hat{\eta}} S^0
\]
has degree $1$ in the underlying non-equivariant category. This shows that $x^1_1 \mapsto x^1_1$ on $\un{\pi}_0$. If $n$ is odd this suffices. If $n$ is even, we can also determine whether $x^n_1 \mapsto x^n_1+\ldots$ or $x^n_1 \mapsto -x^n_1+\ldots$ by taking geometric $C_n$-fixed points. But after taking geometric fixed points the above composite becomes
\[
 S^0 \to S^0 \sma S^0 \xto{t} S^0 \sma S^0 \to S^0,
\]
which again has degree $1$ and not $-1$.

For $S^\sigma$, the trace has degree $-1$ in the underlying non-equivariant category and degree $1$ after taking geometric $C_n$-fixed points. This shows that $x^1_1 \mapsto -x^1_1$ and $x^n_1 \mapsto x^n_1 + \ldots$, which is exactly what multiplication by $-\tau$ does.

To complete the proof, it suffices to observe that none of the isomorphisms $X_a \to X_{\overline{a}}$ do anything unexpected on $\un{\pi}_0$.
\end{proof}

\section{Cohomological Mackey functors and the Picard group} \label{s:cohomMackey}
Recall that a Mackey functor $\un{M}$ is \emph{cohomological} if $\tr^H_K \circ R^H_K$ is multiplication by the index $[H:K]$ for all subgroups $K \leq H \leq G$.

When restricted to the subcategory of cohomological $G$-Mackey functors, $\Box$ is a symmetric monoidal product with unit $\un{\bZ}$ (rather than $\un{\cA}_G$, which is not cohomological).

\begin{proof}[Proof of Theorem \ref{t:B}]
Suppose $[X]$ is in the kernel of $d : \pic(\Gsp) \to C(G)$ and let $\un{M}$ be a cohomological Mackey functor.

To show that $X \sma H\un{M} \simeq H\un{M}$ it suffices to show that $\un{\pi}_0 (X \sma H\un{M}) \cong \un{\pi}_0(X) \Box \un{M}$ is isomorphic to $\un{M}$. Because $\un{M}$ is cohomological we have $\un{\pi}_0(X) \Box \un{M} \cong \un{\pi}_0(X) \Box H\un{\bZ} \Box \un{M}$, so it suffices to show that $\un{\pi}_0(X) \Box H\un{\bZ}$ is isomorphic to $\un{\bZ}$.

Chooose $Y$ with $X \sma Y \simeq S^0_G$. Then $\un{M}_1 = \un{\pi}_0(X) \Box \un{\bZ}$ is invertible with inverse $\un{M}_2 = \un{\pi}_0(Y) \Box \un{\bZ}$. So it suffices to show that the only invertible $\un{\bZ}$-module is $\un{\bZ}$ itself. In other words, it suffices to show that the Picard group of the category of $\un{\bZ}$-modules is trivial.

First observe that if $\un{M}_1$ is an invertible $\un{\bZ}$-module then $\un{M}_1(G/H) \cong \bZ$ for all $H \leq G$, and $\un{M}_1$ is determined by the restriction maps $R^H_K : \un{M}_1(G/H) \cong \bZ \to \un{M}_1(G/K) \cong \bZ$ for $K \leq H \leq G$. These must satisfy $R^H_K \mid [H:K]$, plus a compatibility condition coming from the formula for $R^H_L \circ \tr_K^H$ for each pair $K,L \leq H$.

We first consider the case when $G$ is a $p$-group. Let $P_1,\ldots,P_n$ denote the index $p$ subgroups. Denote the restriction map $\Mack_G \to \Mack_H$ by $\downarrow^G_H$. By induction we can assume $\downarrow^G_{P_i} \un{M}_1 \cong \un{\bZ}$. Then $R^G_{P_j} \tr_{P_i}^G$ is always multiplication by $p$, so either all the $R^G_{P_i}$ are $\pm 1$ and $\un{M}_1 \cong \un{\bZ}$ or all the $R^G_{P_i}$ are $\pm p$. But if all the $R^G_{P_i}$ are $\pm p$ then that will remain true in $\un{M}_1 \Box \un{M}_2$, so $\un{M}_1$ is not invertible.

Finally we consider an arbitrary finite group $G$ with Sylow subgroups $P_1,\ldots,P_n$ of order $p_i^{r_1},\ldots,p_n^{r_n}$. Again we can assume by induction that $\downarrow^G_{P_i} \un{M}_1 \cong \un{\bZ}$ for each $i$. But then $R^G_{P_i}$ is multiplication by the same integer $N$ for each $i=1,\ldots,n$. Now $N \mid p_i^{r_i}$ for each $i$, so $N = \pm 1$ and again we conclude that $\un{M}_1 \cong \un{\bZ}$.
\end{proof}


\end{document}